\documentclass[11pt]{amsart}
\usepackage{amsthm,enumerate}
\usepackage{amssymb}
\usepackage{color}
\usepackage{mathtools}
\usepackage[bookmarks, bookmarksdepth=2, colorlinks=true, linkcolor=blue, citecolor=blue, urlcolor=blue]{hyperref}

\usepackage[utf8]{inputenc}
\usepackage[T1]{fontenc}
\usepackage{lmodern}
\usepackage{cleveref}
\usepackage{hyperref}
\usepackage{algorithm}
\usepackage{comment}
\usepackage{url, mathrsfs}
\usepackage{amsmath}
\usepackage{graphicx, 
epstopdf}
\usepackage{subfig}
\usepackage{color}
\usepackage{bbm}
\usepackage{subfloat}
\usepackage[draft]{fixme}
\usepackage{bm} 
\usepackage{bbm} 
\usepackage{epigraph}
\usepackage{endnotes}
\usepackage{ifpdf}
\usepackage{array}
\usepackage{wrapfig}
\usepackage{longtable}

\usepackage{multicol}
\usepackage{multirow}
\usepackage{graphicx}
\usepackage{tikz}
\usepackage{tikz-cd}
\usepackage{makecell}
\usepackage{fullpage}

\theoremstyle{plain}
\newtheorem{theorem}{Theorem}[section]
\newtheorem{proposition}[theorem]{Proposition}
\newtheorem{corollary}[theorem]{Corollary}
\newtheorem{lemma}[theorem]{Lemma}

\newtheorem*{claim*}{Claim}

\newtheoremstyle{case}{}{}{}{}{}{:}{ }{}
\theoremstyle{case}

\theoremstyle{definition}
\newtheorem{definition}[theorem]{Definition}
\newtheorem{example}[theorem]{Example}
\newtheorem{remark}[theorem]{Remark}

\newcommand \R{{\mathbb R}}

\newcommand \1{{\mathbf 1}}

\newcommand \K{{\mathcal K}}

\newcommand{\psd}{\mathcal{S}^n_+}

\newcommand\bigzero{\makebox(0,0){\text{\huge0}}}

\DeclareMathOperator{\CL}{\mathrm {CLC}}
\DeclareMathOperator{\iSL}{\mathring{SL}}
\DeclareMathOperator{\SL}{\mathrm{SL}}

\DeclareMathOperator{\inter}{int}

\newcommand \x{{\mathbf{x}}}

\def \a{{\mathbf{a}}}
\def \b{{\mathbf{b}}}

\def \a{{\mathbf{a}}}
\def \b{{\mathbf{b}}}

\title{$\K$-Lorentzian Polynomials}
\author{Grigoriy Blekherman and Papri Dey}
\address{School of Mathematics, Georgia Institute of Technology, Atlanta, GA 30332}
\date{}

\begin{document}

\begin{abstract}
Lorentzian polynomials are a fascinating class of real polynomials with many applications. Their definition is specific to the nonnegative orthant. Following recent work, we examine Lorentzian polynomials on proper convex cones. For a self-dual cone $\K$ we find a connection between $\K$-Lorentzian polynomials and $\K$-positive linear maps, which were studied in the context of the generalized Perron-Frobenius theorem. We find that as the cone $\K$ varies, even the set of quadratic $\K$-Lorentzian polynomials can be difficult to understand algorithmically. We also show that, just as in the case of the nonnegative orthant, $\K$-Lorentzian and $\K$-completely log-concave polynomials coincide.
\end{abstract}
\maketitle

\section{Introduction}

Lorentzian polynomials were introduced by Br\"{a}nd\'{e}n and Huh in \cite{Brandenlorentzian} and simultaneously by Anari, Oveis Gharan and Vinzant \cite{Cynthialog} (under the name of completely log-concave polynomials) as a generalization of hyperbolic polynomials that retain their log-concavity properties. They have already seen a multitude of applications \cite{MR4681146, Cynthialog3,MR4232134}. Lorentzian polynomials are specific to the nonnegative orthant $\mathbb{R}^n_{\geq 0}$. Lorentzian polynomials with respect to a full-dimensional convex cone $\K\subset \R^n$ were introduced by Br\"{a}nd\'{e}n and Leake in \cite{Leake} and by the second author. 

We consider what interesting properties of Lorentzian polynomials are maintained when we switch from the nonnegtave orthant to more general cones. The cone $\R^n_{\geq 0}$ is self-dual, and for any self-dual cone $\K$ there is an immediate connection between $\K$-Lorentzian polynomials and well-studied concept of $\K$-nonnegative linear maps \cite{Buckley}. This allows us to also consider properties of $\K$-irreducibility and $\K$-positivity, which have been studied in the context of generalizing the Perron-Frobenius theorem \cite{gantmacher, Vandergraftspectral}.

We take a detailed look at $\K$-Lorentzian quadratic forms for various self-dual cones $\K$. Quadratic Lorentzian polynomials are simple to understand, and so are  quadratic forms Lorentzian on the second order cone $\mathcal{L}_n$ which can be characterized via a simple semidefinite program. However, quadratic forms Lorenztian on the cone of positive semidefinite matrices $\mathcal{S}^n_+$ are more difficult. We show that a natural linear projection of $\mathcal{S}^n_+$-Lorentzian quadratic forms is the convex cone of degree $4$ globally nonnegative polynomials. From this we see that testing whether a quadratic form is $\mathcal{S}^n_+$-Lorentzian is NP-hard.

In analogy with results of \cite{Brandenlorentzian} and \cite{Cynthialog3} we show that $\K$-Lorentzian polynomials are equivalently completely log-concave polynomials (CLC) on $\K$. We also find some interesting properties of Hessians of $\K$-Lorentzian polynomials related to $\K$-nonnegativity and $\K$-irreducibility. We provide a different proof of a result of \cite{Leake} that testing whether a polynomial is $\K$-Lorentzian can be reduced to taking directional derivatives with respect to extreme rays of $\K$ and one interior point of $\K$, and examine the interior of the set of $\K$-Lorentzian polynomials.
The rest of the paper is structured as follows: in \Cref{sec:defnot} we set notation and main definitions. \Cref{sec:QKL} is devoted to a detailed examination of quadratic $\K$-Lorentzian polynomials; we also make a first connection to $\K$-nonnegative and $\K$-irreducible matrices. In \Cref{sec:equivalence} we examine the interior of the set of $\K$-Lorentzian polynomials, show equivalence between $\K$-Lorentzian and $\K$-completely log-concave polynomials and provide a sufficient condition for a sum of two $\K$-Lorentzian polynomials to be $\K$-Lorentzian. In \Cref{sec:selfdual} we expand on the connection to $\K$-irreducibility and $\K$-nonnegativity and show a simplified characterization of $\K$-Lorentzian polynomials.

\section{Definitions and Notation} \label{sec:defnot}

A nonempty convex set $\K \subseteq \R^{n}$ is said to be a cone if $c \K \subseteq \K$ for all $c \geq 0$. A cone $\K$ is called proper if it is  closed (in the Euclidean topology on $\R^{n}$), pointed (i.e.,
$\K \cap (-\K) = \{0\}$), and solid (i.e., the topological interior of $\K$, denoted as $\inter \K$, is nonempty). 

The vector space of real homogeneous polynomials (forms) in $n$ variables of degree $d$ is denoted by $\R[x]^d_n$. A form $f \in \R[x]_n^d$ is said to be \emph{log-concave} at a point $a\in \R^n$ if $f(a)>0$ and $\log f$ is a convex function at $a$, i.e. the Hessian of $\log f$ is negative semidefinite at $a$. 
A form $f$ is \emph{strictly log-concave} at $a$ if the Hessian of $\log f$ at $a$ is negative definite. A form $f \in \R[x]_n^d$ is log-concave on a proper cone $\K\subset \R^n$ if $f$ is log-concave at every point of the interior of $\K$. By convention, the zero polynomial is log-concave (but not strictly log-concave) at all points of $\R^n$.

We will denote quadratic forms by a lower case letter, and the matrix of the quadratic form by the corresponding upper case letter, i.e. if $q$ is a quadratic form, then its matrix is $Q$ and $q(x)=x^tQx$.

For a point $a\in \R^n$ and $f\in \R[x]_n^d$ we use $D_a f$ to denote the directional derivative of $f$ in direction $a$: $D_a f=\sum_{i=1}^n a_i\frac{\partial f}{\partial x_i}$. We now define $\K$-completely log-concave and $\K$-Lorentzian forms.

\begin{definition} \label{def:clc}
A form $f \in \R[x]_n^d$ is called a $\K$-completely log-concave ($\K$-CLC) on a proper convex cone $\K$ if for any choice of $a_{1}, \dots, a_{m} \in \K$, with $m\leq d$, we have that $D_{a_{1}} \dots D_{a_m} f$ is log-concave on $\inter \K$.  
A form $f \in \R[x]_n^d$ is strictly $\K$-CLC if for any choice of $a_{1}, \dots, a_{m} \in \K$, with $m\leq d$,  $D_{a_{1}} \dots D_{a_m} f$ is strictly log-concave on all points of $\K$. We use $\CL^d(\K)$ to denote the set of $\K$-CLC forms in $\R[x]_n^d$.
\end{definition}

\begin{definition} \label{def:pls}
Let $\K$ be a proper convex cone. A form $f \in \R[x]_n^d$ of degree $d\geq 2$ is said to be $\K$-Lorentzian if for any $a_{1},\dots,a_{d-2} \in  \inter \K$  the quadratic form $q=D_{a_{1}} \dots D_{a_{d-2}}f$ satisfies the following conditions:  
\begin{enumerate}
 \item  The matrix $Q$ of $q$ has exactly one positive eigenvalue.
\item  For any $x,y \in\inter \K$ we have $y^t Q x=\langle y,Qx\rangle>0$.
\end{enumerate}
For degree $d \leq 1$ a form is $\K$-Lorentzian if it is nonnegative on $\K$. We use $\SL^d(\K)$ to denote the set of $\K$-Lorentzian forms in $\R[x]_n^d$. 
\end{definition}
\begin{remark} \label{remark:defequiv}
\Cref{def:pls} is equivalent to the definition in \cite{Leake}, although we state it in a slightly different way. For a quadratic form $q$ with matrix $Q$ and for all $a,b\in \R^n$ we have $D_bD_a q=2b^tQa$. Therefore, we can also express condition (2) as $D_bD_aq>0$ for all $a,b \in \inter \K$. Combining with the fact that we have the freedom to choose directions of derivatives $a_1,\dots, a_{d-2} \in \inter \K$ we can express condition (2) as $D_{a_{1}} \dots D_{a_{d}} f>0$ for all
$a_{1},\dots,a_{d} \in  \inter \K$, as was done in \cite{Leake}.

We can also write conditions (1) and (2) above without relying on writing the matrix of $q$ in coordinates. Condition (1) can be stated via the variational definition of eigenvalues: there exists a point $x\in \R^n$ such that $q(x)>0$, and for any two dimensional subspace $V$ of $R^n$, there exists $v\in V$, $v\neq 0$ such that $q(v)\leq 0$. Condition (2) can also be stated in terms of the naturally associated bilinear form $B(x,y)$ of $q(x)$: there exists a unique bilinear form $B(x,y)$ such that $q(x)=B(x,x)$ for all $x\in \R^n$. We can equivalently say that $B(x,y)>0$ for all $x,y \in \inter \K$. 
\end{remark}

\section{Quadratic \texorpdfstring{$\K$}--Lorentzian polynomials} \label{sec:QKL}
In this section we analyze properties of quadratic $\K$-Lorentzian polynomials. Although these observations are mostly elementary, they point to interesting connections with other well-studied notions. We first describe several equivalent formulations for nonnegativity condition (2) in \Cref{def:pls}, which may be easier to verify.

\begin{lemma}\label{lem:red1}
A quadratic form $q(x)\in \R[x]_n^2$ is $\K$-Lorentzian on a proper convex cone $\K$ if and only if the matrix $Q$ of $q$ has exactly one positive eigenvalue and one of the following conditions holds:
\begin{enumerate}[(a)]
\item for all $x\in \K$ we have $Qx\in \K^*$, i.e. $Q(\K)\subseteq{\K^*}$.
     \item $y^tQx \geq 0$ for all $x,y$ spanning extreme rays of $\K$.
    \item $x^tQx > 0$ for all $x \in \inter \K$.
\end{enumerate}
   
\end{lemma}

\begin{remark}
Note that conditions (b) and (c)  of Lemma \ref{lem:red1} are relaxations of Condition (2) from \Cref{{def:pls}}.
\end{remark}

\begin{proof}
(a) By the Biduality Theorem \cite[Chapter-IV]{Barvinok} $\langle{y,Qx}\rangle \geq 0$ for all $y\in \K$ if and only if $Qx\in \K^*$. If $Q(\K)=\{0\}$, then $Q$ is the zero matrix and $q$ is $\K$-Lorentzian. If $Q$ is not the zero matrix and $Qx=0$ for $x\in\inter \K$, then $Q(\K)$ contains a linear subspace. Since $Q(\K)\subseteq \K^*$ we see that $\K$ is not full-dimensional, which is a contradiction. Therefore $Qx\neq 0$ for any $x\in \inter \K$. Then for any $y\in \inter K$ we have $\langle y, Qx\rangle>0$.\\ \smallskip
(b) Since $\K$ is a proper convex cone, by the conical Minkowksi Theorem \cite[Chapter II]{Barvinok}, it is the conical hull of its extreme rays. Therefore $Q(\K)\subseteq \K^*$ if and only if for all $x \in \K$ spanning extreme rays of $\K$ we have $Qx\subset \K^*$. Additionally, we have $Qx\in \K^*$ if and only if for all $y$ spanning extreme rays of $\K$ we have $y^tQx\geq 0$.\\ \smallskip
(c) Let $v=(v_1,\dots,v_n) \in \inter \K$. Let $C$ be the connected component of the set $\{x\in \R^n \,\, \mid \,\, q(x)>0 \}$ which contains $v$. Then $C$ is an open convex cone \cite{Gaarding}, \cite{Gulerhypint} (which is the hyperbolicity cone of $q$), and its closure $\bar{C}$ is a closed, convex cone. Note that we have $\K \subset \bar{C}$, and we need to show that $Q(\K)\subseteq \K^*$. Since duality reverses inclusion, we have $\bar{C}^*\subseteq \K^*$. Therefore it suffices to show that  $Q(\bar{C})\subseteq \bar{C}^*$. 

Since $Q$ has exactly one positive eigenvalue, then up to an orthogonal change of variables, we can write $Q(x)=\lambda_1 x_1^2-\lambda_2 x_2^2-\dots-\lambda_k x_k^2$ with $k\leq n$ and $\lambda_i>0$. Furthermore we may assume $v_1>0$. Let $q'$ be the quadratic form $q'(x)=\frac{1}{\lambda_1} x_1^2-\frac{1}{\lambda_2} x_2^2-\dots-\frac{1}{\lambda_k} x_k^2$. We have 
\begin{align*}
&\bar{C}=\{(x_1,\dots,x_n) \in \R^n \,\, \mid \,\, x_1\geq 0,\, q(x) \geq 0 \},\\
&\bar{C}^*=\{(x_1,\dots,x_n) \in \R^n \,\, \mid \,\, x_1\geq 0,\, q'(x) \geq 0, \, x_{k+1}=\dots=x_n=0 \}.
\end{align*}
Take $x=(x_1,\dots,x_n) \in \bar{C}$. Let $y=Qx=(\lambda_1x_1,-\lambda_2x_2,\dots,-\lambda_kx_k,0,\dots,0)$. Observe that $y_1\geq 0$, $y_{k+1}=\dots=y_n=0$ and $q'(y)=q(x) \geq 0$. Therefore $y\in \bar{C}^*$ and the Lemma follows. 

\end{proof}

Next we show that nonnegativity condition (2) from \Cref{def:pls} can be even further relaxed to just testing nonnegativity of $q$ on the extreme rays of $\K$ in case the matrix $Q$ is nonsingular, and the set of extreme rays of $\K$ is path-connected. 

\begin{lemma}\label{lem:red2}
Let $q$ be a a quadratic form on $\R^n$ such that its matrix $Q$ is nonsingular and has exactly one positive eigenvalue. Let $\K$ be a proper convex cone in $\R^n$ such that the set of vectors spanning an extremal ray of $\K$ is path-connected. Then $q\in \SL^2(\K)$ if and only if $q(v)\geq 0$ for all $v$ spanning an extremal ray of $K$.
\end{lemma}

\begin{remark}
The reduction to extreme rays is not possible for the nonnegative orthant $\R^n_{\geq 0}$ (the standard setting for Lorentzian polynomials), or for any other polyhedral cone $\K$. In these cases, the set of vectors spanning extreme rays of $\K$ is not connected. 
\end{remark}
\begin{proof}
Up to a change of coordinates we have $q(x)=x_1^2-x_2^2-\dots-x_n^2$. Let $v,w$ be two points spanning distinct extreme rays of $\K$. We see that the first coordinates $v_1$, $w_1$ must have the same sign, since there is a path connecting $v$ and $w$ which does not pass through the origin, such that for any point $t$ on the path we have $q(t)\geq 0$. Without loss of generality we may assume that for any $v$ spanning an extremal ray of $\K$ we have $v_1>0$. It follows that all extreme rays of $\K$ are contained in the proper convex cone $\bar{C}=\{(x_1,\dots,x_n) \in \R^n \,\, \mid \,\, x_1\geq 0,\, q(x) \geq 0 \}$. Therefore, the interior of $\K$ is contained in the interior of $\bar{C}$ and $q$ is strictly positive on the interior of $\bar{C}$. The Lemma now follows by Lemma \ref{lem:red1}.
\end{proof}

\begin{remark}
The hypothesis that $Q$ is nonsingular cannot be eliminated. Let $\ell$ be any linear form such that $\ell(v)=0$ for some $v\in \inter \K$, and let $q=\ell(v)^2$. Then $Q$ has exactly one positive eigenvalue, and $q$ is nonnegative on all the extreme rays of $\K$, but $q$ is not $\K$-Lorentzian. 
\end{remark}

Let $\K$ be a self-dual cone, i.e. $\K=\K^*$. We can use Condition (a) of \Cref{lem:red1} to make a connection between $\K$-Lorentzian quadratic forms and $\K$-nonnegative matrices (or equivalently linear maps).

\begin{definition} \label{def:k-irre}
Let $\K$ be a proper convex cone in $\R^{n}$. A $n\times n$ matrix $A$ is $\K$-nonnegative if $Ax \in \K$ for any $x \in \K$, i.e. $A(\K)\subseteq \K$. A matrix $A$ is $\K$-positive if $Ax \in \inter \K$ for any nonzero $x \in \K$. A $\K$-nonnegative matrix $A$ is called $\K$-irreducible if $A$ leaves no (proper) face
of $\K$ invariant. A $\K$-nonnegative matrix which is not $\K$-irreducible is called $\K$-reducible, cf. \cite{gantmacher}, \cite{Vandergraftspectral}.
\end{definition}
Here we record some equivalent characterizations to $\K$-irreducibility which will be used in the sequel  \cite{Vandergraftspectral}.
\begin{theorem} \label{them:irreducible} Let $A$ be a $\K$-nonnegative matrix for a solid (with nonempty interior) cone $\K$. The following statements are equivalent:
    \begin{enumerate}
    \item $A$ is $\K$-irreducible
    \item no eigenvector of $A$ lies on the boundary $\partial \K$ of $\K$.
    \item $A$ has exactly one eigenvector in $\K$, and this eigenvector is in $\inter \K$.
    \end{enumerate}
\end{theorem}
From \Cref{lem:red1} (a) we immediately see the following:

\begin{corollary}\label{cor:kpos}
    Let $\K$ be a self-dual cone. Then $q\in \SL^2(\K)$ if and only if the matrix $Q$ of $q$ has exactly one positive eigenvalue and $Q$ is $\K$-nonnegative.
\end{corollary}

We now show that if a matrix $Q$ of $q\in \SL^2(\K)$ is nonsingular and the cone $\K$ is self-dual, then $Q$ must be $\K$-irreducible.

\begin{lemma} \label{lemma:clcirre}
Let $\K \subset \R^{n}$ be a self-dual cone and $q \in \SL^{2}(\K)$ with matrix $Q$. If $Q$ is nonsingular, then $Q$ is $\K$-irreducible. 
\end{lemma}
\begin{proof}
It suffices to show that a symmetric, nonsingular $\K$-nonnegative matrix $Q$ with exactly one positive eigenvalue must be $\K$-irreducible.
Let $v$ be the eigenvector associated with the positive eigenvalue $\alpha$ of $Q$. Suppose that neither $v$ nor $-v$ lie in the interior of $\K$. Since $\K$ is self-dual, there exists a nonzero vector $w\in \K$ such that $\langle v,w\rangle=0$. Since $Q$ is nonsingular, we see that the quadratic form $q$ is negative definite on the orthogonal complement of $v$, and in particular, we have $q(w)<0$. This is a contradiction, since $q$ is nonnegative on $\K$. The Lemma now follows by \Cref{them:irreducible}. \end{proof}
Combining \Cref{cor:kpos} and \Cref{lemma:clcirre} we immediately obtain the following:
\begin{corollary} \label{cor:qclcirre}
Let $\K\subset \R^n$ be a self-dual cone. Then $q\in \SL^2(\K)$ if and only if the matrix $Q$ of $q$ has exactly one positive eigenvalue and $Q$ is either nonsingular and $\K$-irreducible, or singular and $\K$-nonnegative.
\end{corollary} 

\subsection{Quadratic Lorentzian Polynomials on \texorpdfstring{$\R^n_{\geq 0}$}{} and the Second order cone.} Quadratic forms Lorentzian on the nonnegative orthant $\R^n_{\geq 0}$ (the usual Lorenztian polynomials) are simply quadratic forms with nonnegative coefficients and exactly one positive eigenvalue (see e.g. \Cref{lem:red1} (b)). Therefore it is easy to test whether a quadratic form is $\R^n_{\geq 0}$-Lorentzian. It was recently shown in \cite{chin2024real} that while testing whether a polynomial is Lorentzian is easy in any degree, testing whether a polynomial is stable (hyperbolic on the nonnegative orthant) is a diffciult computational problems starting in degree $3$.

The second order cone $\mathcal{L}_n\subset \R^n$ is defined by $$\mathcal{L}_n=\{(x_1,\dots,x_n)\in \R^n \,\, \mid \,\, x_1^2\geq x_2^2+\dots+x_n^2 \}.$$ 
Let $B$ be the diagonal matrix $n\times n$ matrix with $1$ in the $(1,1)$-entry and $-1$ in all other diagonal entries. It is well-known that by the S-Lemma \cite{Polik2007survey}, \cite{Loewyice} a quadratic form $q$ with matrix $Q$ is nonnegative on $\mathcal{L}_n$ if and only if there exists a nonnegative $\lambda \in \R$ such that $Q-\lambda B$ is positive semidefinite. Note that $\lambda>0$ would imply that $q$ is strictly positive on the interior of $\mathcal{L}_n$, and therefore if $q$ has exactly one positive eigenvalue, then $q$ is $\mathcal{L}_n$-Lorentzian (\Cref{lem:red1} (c)). If $\lambda=0$, then $Q$ is positive semidefinite, and if $Q$ has exactly one positive eigenvalue, then $q=\langle v,x\rangle^2$ for some $v\in \R^n$. In this case $q$ is $\mathcal{L}_n$-Lorentzian if and only if $v \in \mathcal{L}_n$, since the second order cone is self-dual. Therefore we can test whether a quadratic form is $\mathcal{L}_n$-Lorentzian by semidefinite programming and establishing the signature of $Q$.

However, as we will see below, there are self-dual cones $\K$ for which the set of $\K$-Lorenztian quadratic forms is complicated. We now turn to the cone of positive semidefinite matrices $\psd$.

\subsection{Quadratic Lorentzian Polynomials on \texorpdfstring{$\psd$}{}} 
We consider quadratic forms Lorentzian on the cone $\mathcal{S}^n_+$ of $n\times n$ positive semidefinite symmetric matrices. Let $\mathcal{S}^n$ denote the vector space of real symmetric $n\times n$ matrices equipped with the standard inner product $\langle A,B \rangle =\operatorname{trace} AB$. It is well-known that the cone $\mathcal{S}^n_+$ is self-dual with the respect to the standard inner product \cite[Chapter II.12]{Barvinok}. Linear maps that are $S^n_+$-nonnegative are known as \emph{positive linear maps}, and they are also related to nonnegative biquadratic forms on $\mathbb{R}^n\times \mathbb{R}^n$. We describe this connection now.

Let $V$ be the vector space of linear maps from $\mathcal{S}^n$ to itself. 
Note that here we are not restricting the matrices of linear transformations to be symmetric, and we have $\dim V=\binom{n}{2}^2$. A linear map $L:\mathcal{S}^n\rightarrow \mathcal{S}^n$ is called \emph{a positive linear map} if $L(\mathcal{S}^n_+)\subseteq \mathcal{S}^n_+$. It follows from \Cref{cor:kpos} that quadratic $\mathcal{S}^n_+$-Lorentzian polynomials correspond to positive linear maps whose matrix is symmetric and has exactly one positive eigenvalue. 

Let $W$ be the vector of real biquadaratic forms on $\mathbb{R}^n\times \mathbb{R}^n$, i.e. homogeneous polynomials in variables $(x_1,\dots,x_n)$, $(y_1,\dots,y_n)$ which are quadratic in $x$ and $y$ variables. Note that $\dim W=\binom{n}{2}^2$. We can define a linear map $\Psi$ from $V$ to $W$.
$$\Psi: V\rightarrow W, \,\,\, \text{via} \,\,\, \Psi(L)=\langle yy^t, L(xx^t) \rangle.$$

It is well-known that $\Psi$ is a linear isomorphism and a map $L:\mathcal{S}^n\rightarrow \mathcal{S}^n$ is positive if and only if the associated biquadratic form $\Psi(L)$ is nonnegative on $\mathbb{R}^n\times \mathbb{R}^n$ \cite{Buckley}.  

Let $V'$ be the vector space of quadratic forms on $\mathcal{S}^n$, which we can equivalently see as a subspace of $V$ consisting of linear maps $\mathcal{S}^n\rightarrow \mathcal{S}^n$ with a symmetric matrix. It follows that $\Psi(V')$ is the vector space $W'$ of biquadratic forms which are symmetric under swapping $x$ and $y$. We see from the above discussion that the image of $\SL^{2}(\mathcal{S}^n_+)$ under $\Psi$ is a non-convex subset of the cone of nonnegative biquadratic forms in $W'$. It is not clear to us if the signature condition of exactly one positive eigenvalue has a natural translation in terms of biquadratic forms.

 However, we can use the criterion of Lemma \ref{lem:red2} to provide a simpler criterion for nonsingular quadratic forms to be $\mathcal{S}^n_+$-Lorentzian , which will give us a different linear map, where we will have full understanding of the image of $\SL^{2}(\mathcal{S}^n_+)$

\begin{corollary}\label{cor:pos}
   Let $q$ be a quadratic form on $\mathcal{S}^n$ whose matrix $Q$ is nonsingular and has exactly one positive eigenvalue. Then $q$ is $\mathcal{S}^n_+$-Lorentzian if and only if the associated quartic polynomial $q(xx^t)$ is nonnegative for all $x\in \mathbb{R}^n$.
\end{corollary}
\begin{proof}
    The extreme rays of the cone $\mathcal{S}^n_+$ are well-known to be spanned by rank $1$ matrices of the form $xx^t$ for $x\in \R^n$ \cite[Chapter II.12]{Barvinok}. Now we apply \Cref{lem:red2}.
\end{proof}

Motivated by the above result we define the linear map $\Phi: V' \rightarrow \mathbb{R}[x]_n^4$ given by $\Phi(q)=q(xx^t)$. We can define a natural linear projection from the vector space $W$ of biquadratic forms on $\R^n\times\R^n$ to the vectors space of quartic polynomials $\R[x]_n^4$ by sending a biquadratic form $T(x,y)$ to the quartic polynomial $T(x,x)$. We can see that $$\Phi=\pi\circ \Psi,$$ 
i.e. the quartic polynomial $q(xx^t)$ can be obtained by plugging in $y=x$ in the associated biquadratic form $\Psi(q)$. Interestingly, the image of $\SL^{2}(\mathcal{S}_+^n)$ under $\Phi$ is the entire cone $P_{n,4}$ of nonnegative quartics on $\R^n$, and in particular it is convex.

\begin{theorem}\label{thm:proj}
We have $$\Phi(\SL^{2}(\mathcal{S}^n_+))=P_{n,4},$$
and moreover the same holds if we replace $\SL^{2}(\mathcal{S}^n_+)$ with nondegenerate quadratic forms in $\SL^{2}(\mathcal{S}^n_+)$.
\end{theorem}
\begin{proof}
It follows from the above discussion that the image of $\SL^2(\mathcal{S}^n_+)$ under $\Phi$ lies inside $P_{n,4}$. 

Let $r$ be the following quadratic form on $\mathcal{S}^n$: $$r(X)=\sum_{1\leq i<j\leq n} x_{ii}x_{jj}-x^2_{ij},$$
i.e. $r$ is the sum of all $2\times 2$ principal minors of a matrix $X\in \mathcal{S}^n$. It follows that $\Phi(r)=r(xx^t)=0$. If we coordinatize $ \mathcal{S}^n$ by first listing diagonal elements (in any order) and then off-diagonal elements (in any order), then the matrix $R$ of $r$ becomes:
$$R=\begin{pmatrix} \begin{matrix} 0 & 1/2 & \dots & 1/2 & \\ 1/2 & 0 & \dots & 1/2 \\ \vdots & \vdots & \vdots & \vdots \\ 1/2 & 1/2 & \dots & 0  \end{matrix} & \bigzero \\ \bigzero & \begin{matrix} -1 & 0 & \dots  & 0 \\ 0 & -1 & \dots & 0 \\ \vdots & \vdots & \vdots & \vdots \\ 0 & 0 & \dots & -1 \end{matrix} \end{pmatrix}.$$
We see that the matrix $R$ is nonsingular with exactly one positive eigenvalue. Since the map $\Phi$ is onto $\R[x]_n^4$, given a polynomial $p \in P_{n,4}$, there exists a quadratic form $a\in V'$, with matirx $A$, such that $\Phi(a)=p$. Now let $A'=T+tR$ and let $a'$ be the associated quadratic form. For a sufficiently large $t$ we see that $A'$ is nonsingular with exactly one positive eigenvalue, and therefore by \Cref{cor:pos} we have that $a'\in \SL^2(\mathcal{S}^n_+)$ and $\Phi(a')=\Phi(a)=p$.
\end{proof}
It quickly follows from the above result that testing whether a quadratic form is $\mathcal{S}^n_+$-Lorentzian is NP-hard.
\begin{corollary}
    It is NP-hard to test whether a quadratic form is  $\mathcal{S}^n_+$-Lorentzian.
\end{corollary}
\begin{proof}
   Note that the projection $\Phi$ sending a quadratic form $q$ to $q(xx^t)$ simply acts by substituting $x_{ij}=x_ix_j$ into $q$. Under this map, several quadratic monomials $x_{ij}x_{k\ell}$ may map to the same degree 4 monomial, since for example, $x_1^2\cdot x_2^2=(x_1x_2)^2$. Therefore, in coordinates, $\Phi$ simply adds several coefficients of monomials of $q$ to produce a coefficent of the quartic polynomial $\Phi(q)$, and these sums are taken in disjoint groups.
    
    Given a polynomial $p\in \R[x]_n^4$ we can define a canonical preimage $m\in \Phi^{-1}(p)$ as follows. Let $x^a$ be a degree $4$ monomial in $p$ with coefficient $c_a$. Define $d_a$ to be the number of possible quadrartic monomial $x_{ij}x_{k\ell}$ that produce $x^a$ under the substitution $x_{ij}=x_ix_j$. Then coefficient of $x_{ij}x_{k\ell}$ in $m$ is simply $\frac{c_a}{d_a}$ where $a=\Phi(x_{ij}x_{k\ell})$. By construction we have $\Phi(m)=p$.

     Then, as in proof of \Cref{thm:proj} we consider $M'=M+tR$. It is not hard to estimate $t$ such that $M'$ has exactly one positive eigenvalue. Therefore, if we can test in polynomial time whether quadratic form $m'$ associated to $M'$ is $\mathcal{S}^n_+$-Lorentzian, then we can test whether $p \in P_{n,4}$ by \Cref{lem:red2}. However, it is known that testing nonnegativity of degree 4 polynomials is NP-Hard \cite[Part 1]{Blumcomplexity}. 
\end{proof}

\section{Completely \texorpdfstring{$\K$}{}-log-concave and \texorpdfstring{$\K$}{}-Lorentzian polynomials} \label{sec:equivalence}
We now investigate properties of the set $\SL^d(\K)$ of $\K$-Lorentzian polynomials of degree $d$. We consider its interior, and we also show equivalence between $\K$-Lorentzian and $\K$-CLC polynomials.

\subsection{Log-concavity over proper convex cones}
We show several equivalent characterizations of forms log-concave on a proper convex cone $\mathcal{K} \subset \mathbb{R}^{n}$. We employ the same reasoning as outlined in \cite{Cynthialog}. 
Additionally, we utilize the following version of Cauchy's interlacing Theorem.
\begin{lemma} [Cauchy's interlacing Theorem I] \cite[Corollary $4.3.9$]{Hornmatrix} \label{lemma:rank1}
Let $A$ be a symmetric matrix and vector $v \in \R^{n \times n}$. Then, the eigenvalues of $A-vv^{t}$ interlace the eigenvalues of $A$, i.e., if $\alpha_1, \dots , \alpha_n$ are eigenvalues of $A$ and $\beta_1, \dots , \beta_n$
are eigenvalues of $A - vv^T$, then
$$ \beta_1 \leq \alpha_1 \leq \dots \beta_{n} \leq \alpha_n$$
\end{lemma}
\noindent Therefore, we have the following result. 
\begin{lemma} \label{lemma:cauchy} \cite[Lemma 2.4]{Cynthialog}
Let $Q \in \R^{n \times n}$ be a symmetric matrix and let $P \in \R^{m \times n}$. If $Q$ has exactly one positive eigenvalue, then $PQP^{t}$ has at most one positive eigenvalue.
\end{lemma}

\begin{remark}
 Let $Q$ be a symmetric matrix. Then $Q$ has exactly one positive eigenvalue if and only if $Q$ has at most one positive eigenvalue and there exists $a \in \R^{n}$ such that $a^{t}Qa > 0$.
\end{remark}
The following lemma can be found in \cite{Huhenumeration} and \cite{Cynthialog}. However, for the sake of completeness, we provide a proof here to keep the paper self-contained.
\begin{lemma} \label{lemma:negsemdef}
Let $Q$ be a symmetric matrix and $a\in \R^n$ such that $a^tQa>0$. Then the following statements are equivalent:
\begin{enumerate} [(a)]
\item $Q$ has exactly one positive eigenvalue 
\item the matrix $(a^{t}Qa)\cdot Q-t(Qa)(Qa)^{t}$ is negative semidefinite for all $t \geq 1$.
\item $Q$ is negative semidefinite on the hyperplane $(Qa)^{t}$. 
\end{enumerate}
\end{lemma}
\begin{proof} $(a \Rightarrow b)$ Since the matrix $(Qa)(Qa)^{t}$ is positive semidefinite, it suffices to prove part (2) for $t=1$.  
Suppose $b \in \R^{n}$ is not a scalar multiple of $a$. Consider the $2 \times n$ matrix $P$ with rows $a^{t}$ and $b^{t}$. Then 
\begin{equation} \label{eq:negsemdef}
    PQP^{t}=\begin{bmatrix} a^{t}Qa & a^{t}Qb \\b^{t}Qa & b^{t}Qb & \end{bmatrix}
\end{equation}
By \Cref{lemma:cauchy}, the matrix $PQP^{t}$ has at most one positive eigenvalue. Then, $\det(PQP^{t})=a^{t}Qa \cdot b^{t}Qb-a^{t}Qb \cdot b^{t}Qa \leq 0$. Thus, $$b^{t}\left((a^{t}Qa)\cdot Q-(Qa)(Qa)^{t}\right)b \leq 0$$ for all $b \in \R^{n}$. 

($b \Rightarrow c$) Since $a^{t}Qa \cdot Q-(Qa)(Qa)^{t}$ is negative semidefinite and $a^{t}Qa >0$, we have $$b^{t}\left((a^{t}Qa)\cdot Q-(Qa)(Qa)^{t}\right)b \leq 0$$ for all $b \in \R^{n}$. Therefore $a^{t}Qa \cdot b^{t}Qb-a^{t}Qb \cdot b^{t}Qa \leq 0$. Thus, $Q$ is negative semidefinite on the hyperplane $\{b \in \R^{n}\, \mid \, b^{t}Qa=0\}=(Qa)^{\perp}$. 

($c \Rightarrow a$) Since $Q$ is negative semidefinite on the hyperplane $\{b \in \R^n \, \mid \, b^TQa = 0\}=(Qa)^{\perp}$ and $a^{t}Qa>0$, we see that $Q$ has exactly one positive eigenvalue.
\end{proof}

We now present several equivalent conditions for log-concavity of a form $f \in \R[x]_n^d$ on a proper convex cone $\mathcal{K}$, which are quite similar to the conditions described in \cite{Cynthialog} and \cite{Brandenlorentzian} for the nonnegative orthant. 
\begin{proposition} \label{prop:slc}
Let $f \in \R[x]_n^d$ be a form of degree $d \geq 2$ and $\K\subset \R^n$ be a proper convex cone. Suppose that $f(a)>0$ for all $a\in \inter \K$, and
let $Q_a$ denote the Hessian matrix of $f$ evaluated at a:  $Q_a=H_{f}(a)$. The following statements are equivalent:
\begin{enumerate} [(a)]
    \item $f$ is log-concave on $\K$.
    \item $Q_a$ has exactly one positive eigenvalue for all $ a \in \inter \K$. 
  \item The quadratic form $x^{t} Q_ax$ is negative semidefinite on the hyperplane $(Q_aa)^{\perp}$ for all $a\in \inter \K$.
  \item For all $a\in \inter \K$, the matrix $(a^{t}Q_a a)Q_a-(Q_a a)(Q_a a)^{t}$ is negative semidefinite.
\end{enumerate}
\end{proposition}
\begin{proof} Euler's Identity states that for a homogeneous polynomial $f$ of degree $d$, $\langle \nabla f(x) , x \rangle =\sum_{i=1}^{n}x_{i} \partial_{i}f(x)=d \cdot f(x)$. Using this on $f$ and $\partial_{i} f$ we easily get the following identities:
\begin{equation*} 
Q_a a=(d-1) \cdot \nabla f(a), \ a^{t} Q_a a =d(d-1) \cdot f(a).
\end{equation*}
Additionally, the Hessian of $\log(f)$ at $x=a$ equals
\begin{equation}\label{eq:equiv}
    H_{\log f}(a)=\left( \frac{f \cdot H_{f}-\nabla f \nabla f^{t}}{f^{2}}\right)|_{\x=\a}=d(d-1) \frac{a^{t}Q_a a \cdot Q_a-\frac{d}{d-1}(Q_a a)(Q_a a)^{t}}{(a^{t}Qa)^{2}}.
\end{equation}
($a \Rightarrow b$) Since $f$ is a non-zero form which is log-concave on all points of the interior of $\K$ we see that $f(a)>0$ for all $a\in \inter \K$, and the 
 Hessian of $\log(f)$ at $a$ is negative semidefinite. Since $a^tQ_a a>0$ we see from \eqref{eq:equiv} that $Q_a$ has at most $1$ positive eigenvalue. However, $Q_a$ cannot be negative definite, since $a^tQ_a a>0$. Thus $Q_a$ has exactly one positive eigenvalue.
 
\noindent ($b \Leftrightarrow c \Leftrightarrow d$) 
The proof follows by \Cref{lemma:negsemdef}. 
\end{proof}

We also have analogous conditions for strict log-concavity.
\begin{proposition} \label{prop:slc2}
Let $f \in \R[x]_n^d$ be a form of degree $d \geq 2$ and $\K\subset \R^n$ be a proper convex cone. Suppose that $f(a)>0$ for all $a\in \K$, and
let $Q_a$ denote the Hessian matrix of $f$ evaluated at a:  $Q_a=H_{f}(a)$. The following statements are equivalent:
\begin{enumerate} [(a)]
    \item $f$ is strictly log-concave on $\K$. 
     \item For all $a \in \K$, the quadratic form $x^t  Q_a x$ is negative definite on the hyperplane $(Q_aa)^{\perp}$ . 
    \item The matrix $Q_a$ is nonsingular and has exactly one positive eigenvalue for all $a \in \K$. 
\end{enumerate}
\end{proposition}

\begin{proof}
\noindent  ($a \Rightarrow b$) The Hessian of $\log( f(x))$ at $x = a$ is negative definite. Then by Equation \eqref{eq:equiv} the matrix $M_a:=(a^{t}Q_a a)\cdot Q_a-\frac{d}{d-1}(Q_a a)(Q_a a)^{t}$ is negative definite. Consider the restriction of the quadratic form $x^tM_ax$ to the hyperplane $(Q_aa)^{\perp}$. For $x$ in this hyperplane we have $x^tM_ax=(a^{t}Q_a a)\cdot x^tQ_ax$, and so $x^tQ_a x$ is negative definite on $(Q_aa)^\perp$.

\noindent ($b \Rightarrow c$) Since $x^tQ_a x$ is negative definite on the hyperplane $(Q_aa)^{\perp}$, we see that $Q$ has at least $n-1$ negative eigenvalues. But, since 
 $d(d-1) \cdot f(a)=a^tQ_a a >0$ we see that $Q_a$ must also have a positive eigenvalue. 
 
 \noindent ($c \Rightarrow a$) We need to establish that the matrix $a^{t}Q_a a \cdot Q_a-\frac{d}{d-1}(Q_a a)(Q_a a)^{t}$ is negative definite using the fact that $Q_a$ is nonsingular. The proof is exactly the same as the proof of $(a \Rightarrow b)$ in \Cref{lemma:negsemdef}.

\end{proof}

The following result shows that $f$ is log-concave at $a$ if and only if its directional derivative in the direction $a$ is log-concave at $a$, cf. \cite[Lemma 2.1]{Cynthialog3}. 
\begin{proposition} \label{prop:logderivative}
Let $f \in \R[x]_n^d$ be a homogeneous polynomial of degree $d \geq 3$ with $f(a)>0$ at some point $a \in \R^{n}$.  Then $f$ is log-concave at $a$ if and only if $D_{a}(f)$ is log-concave at $a$. 
\end{proposition}
\begin{proof} By using Euler's identity for a homogeneous polynomial of degree $d$, we can express $d\cdot f(x)=\sum_{i=1}^{n}x_{i} \frac{\partial f}{\partial x_{i}}=D_{x}(f(x))$. Therefore, $f(a) >0$ if and only if $D_{a}f(a) >0$. Moreover, $H_{D_af}(a)=(d-2)H_f(a)$. Then at the point $x=a$, signature of $H_{f}(a)$ and $H_{D_{a}f}(a)$ are the same. Then the claim follows from \Cref{prop:slc}. \end{proof}

We also show that log-concavity at a point is preserved under invertible linear transformations. 
\begin{lemma} \label{lemma:coordinatechange}
If $f(x)\in \R[x]_n^d$ 
is log-concave at $a$, then $f(Ax)$ is log-concave at $A^{-1}a$ for any nonsingular $n \times n$ matrix $A$.
\end{lemma}
\begin{proof} The Hessian of $f(Ax)$ at $A^{-1}a$ is equal to  $A^tH_f(a)A$. Therefore the Hessian of $f(Ax)$ at $A^{-1}a$ has exactly one positive eigenvalue and log-concavity of $f(Ax)$ at $A^{-1}a$ follows by \Cref{prop:slc}.
\end{proof}

\subsection{Interior}\label{sec:int}

Br\"{a}nd\'{e}n and Leake \cite{Leake} proved that $\SL^d(\K)$ is a closed set, and it is the closure of its interior.  
We use their result to describe the interior of $\SL^d(\K)$. We define the set $\iSL^d(\K)$ as the set consisting forms in $\SL^d(\K)$ which additionally satisfy 
the following properties: for all $a_1,\dots, a_{d-2} \in \K$ the matrix $Q$ of $q=D_{a_1}\dots D_{a_{d-2}}f$ is nonsingular and $x^t Q y>0$ for all $x,y\in \K$, i.e. strict positivity in \Cref{def:pls} holds on all of $\K$, and not just in the interior of $\K$. Next we show that $\iSL^d(\K)$ is the interior of $\SL^d(\K)$.

\begin{theorem} \label{them:plsbdd} Let $\K\subset \R^n$ be a proper convex cone. Then $\iSL^{d}_{n}(\K)$ is the interior of $\SL^{d}_{n}(\K)$, and 
\begin{eqnarray*}
\SL^{d}_{n}(\K)=\overline{\iSL^{d}_{n}(\K)} 
\end{eqnarray*} 
\end{theorem}
\begin{proof} 
It suffices to show that (1) the set $\iSL^{d}_{n}(\K)$ is open, and (2) any form $f\in  \SL^{d}_{n}(\K) \setminus \iSL^{d}_{n}(\K)$ does not lie in the interior of $\SL^{d}_{n}(\K)$.

Openness of $\SL^{d}(\K)$ is clear for $d\leq 1$. For $d \geq 2$, let $f\in \iSL^d(\K)$ and let $B$ be the intersection of $\K$ with the unit sphere in $\R^n$: $B=\mathbb{S}^{n-1}\cap \K$. Let $\alpha$ be the infimum of $x^t Q y$ where $q=D_{a_1}\cdots D_{a_{d-2}}f$ with $a_i,x,y \in B$. Since we are taking the infimum of a continuous function on a compact set $B^d$ ($d-2$ choices for $a_i$ and one each for $x$ and $y$), we see that $\alpha>0$. Otherwise the minimum of $0$ is attained on some point of $B^d$, which contradicts $f \in \iSL^d(\K)$. Similarly let $\beta$ be the minimum of $| \det Q|$ over all choices of $a_1,\dots,a_{d-2} \in B$. Again we see that $\beta >0$. 
Therefore for any $g \in \R[x]_n^d$ there exists $\varepsilon >0$ such that $f+\varepsilon g \in \iSL^d(\K)$. Thus we see that $\iSL^d(\K)$ is open.

If a form $f$ lies in $\SL^d(\K) \setminus \iSL^d(\K)$, then for some choice of $a_1,\dots,a_{d-2}, x,y \in B$ we have $x^tQy=0$, or $Q$ is singular, where $q=D_{a_1}\cdots D_{a_{d-2}}f$. In the first case, take $g\in \R[x]_{n}^d$ such that $x^t Q'y<0$ where $q'=D_{a_1}\cdots D_{a_{d-2}}g$. Then for any $\varepsilon >0$ we have $f+\varepsilon g \notin \SL^d(\K)$, and thus $f\notin \inter \SL^d(\K)$. In the second case take $g\in \R[x]_{n}^d$ such that $q'=D_{a_1}\cdots D_{a_{d-2}}g$ is strictly positive on the kernel of $Q$. Then for sufficiently small $\varepsilon >0$ we have $f+\varepsilon g \notin \SL^d(\K)$, and thus $f\notin \inter \SL^d(\K)$.
\end{proof}

\subsection{Equivalence of \texorpdfstring{$\K$}{}-Lorentzian and \texorpdfstring{$\K$}{}-CLC forms} 
In this subsection, we will often utilize the equivalent defintion of $\K$-Lorentzian polynomials using positivity of $D_{a_{1}} \dots D_{a_{d}} f$ for all $a_1,\dots,a_d \in \inter \K$, as explained in \Cref{remark:defequiv}.

\begin{theorem} \label{them:plsclc}
Let $\K\subset \R^n$ be a proper cone. A form $f\in \R[x]_n^d$ is $\K$-Lorentzian if and only if $f$ is $\K$-CLC. Equivalently, $\SL^d(\K)=\mathrm{CLC}^d (\K)$.
\end{theorem}
\begin{proof}  $\Leftarrow$  
For $f \in \mathrm{CLC}^{d}(\K)$, the form $f$ and all its repeated directional derivatives are log-concave on $\inter \K$. Thus we see that, the conditions for a polynomial to lie in $\SL^d(\K)$ are a subset of the conditions for a polynomial to lie in $\mathrm{CLC}^d(\K)$.

$\Rightarrow$
According to \Cref{def:pls}, nonnegative constants and linear forms nonnegative on $\K$ are $\K$-Lorentzian polynomials. We establish this result by induction on the degree $d$ of the forms. For the base case $d=2$, let $f\in \SL^2(\K)$. The positivity condition (2) in \Cref{def:pls} and \Cref{remark:defequiv} ensure log-concavity for directional derivatives of degree $1$ and $0$ for $f$. The signature condition (1) in \Cref{def:pls} implies that $f$ is log-concave on $\inter \K$ by \Cref{prop:slc}.

Now we assume that the result holds for forms in $\SL^{d-1}(\K)$. 
 By the induction hypothesis, $D_{a}f$ is $\K$-CLC for any $a \in \inter \K$, and since the set $\SL^{d-1}(\K)=\CL^{d-1}(\K)$ is closed, we see that $D_{a}f$ is $\K$-CLC for any $a \in \K$. Thus it suffices to show that $f$ is log-concave on $\inter \K$. Similarly to the base case, the positivity condition (2) in \Cref{def:pls} implies $f(a)>0$ for all $a \in \inter \K$.
By the induction hypothesis, $D_af$ is log-concave at any $a \in \inter \K$, and we also have $H_{D_af}(a)=(d-2)H_{f}(a)$. Therefore, by \Cref{prop:logderivative} we see that $f$ is log-concave at $a$. 
 Thus, $f\in \CL^d(\K)$.
\end{proof}

\begin{remark} \label{remark:int}
    By \Cref{them:plsclc} and \Cref{them:plsbdd} we have shown that $\CL^{d}(\K)$ is the closure of its interior, denoted as $\inter \CL^{d}(\K)$, consisting of strictly $\K$-CLC polynomials.
\end{remark}

\subsection{\texorpdfstring{$\K$}{}-CLC forms over proper convex cones}

We now derive some consequences for the structure of the set $\CL^d(\K)=\SL^d(\K)$. We use $H_f$ to denote the Hessian matrix of a form $f$.
\begin{lemma} \label{lemma:nonsingHes}
Let $f \in \R[x]^d_n$ be a nonzero polynomial and $\K$ be a proper convex cone. 
\begin{enumerate}[(a)]
\item If $f \in \CL^{d}(\K)$, then $H_{f}(a)$ has exactly one positive eigenvalue for all $a \in \inter \K$.
\item If $f \in \inter \CL^{d}(\K)$, then $H_{f}(a)$ is nonsingular and has exactly one positive eigenvalue for all $a \in \K$.
\end{enumerate}
\end{lemma}
\begin{proof} $(a)$ follows from \Cref{prop:slc}. \\
($b$)  $f \in \inter \CL^{d}(\K)$ implies that $f$ is strictly log-concave on all points of $\K$. The rest follows from \Cref{prop:slc2}. 
\end{proof}
\begin{proposition} \label{prop:mslc}
Let $f \in \R[x]_n^d$ be a nonzero form of degree $d \geq 2$, $\K$ be a proper convex cone in $\R^{n}$, and $Q_a=H_{f}(a)$. 
\begin{enumerate} [(a)]
\item Let $f \in \inter \CL^{d}(\K)$. Then quadratic form $x^{t} Q_ax$ is negative definite on $(Q_ab)^{\perp}$ for every $a, b \in  \K$ such that $Q_ab \neq 0$. 
\item Let $f \in \CL^{d}(\K)$. Then quadratic form $x^{t} Q_ax$ is negative semidefinite on $(Q_ab)^{\perp}$ for every $a \in \inter \K$, and $b \in \K$ such that $Q_ab \neq 0$. 
\end{enumerate}
\end{proposition}
\begin{proof} (a) Given that $f \in \inter \CL^{d}(\K)$ and $a \in \K$, we know that $Q_a$ is nonsingular has exactly one positive eigenvalue by \Cref{lemma:nonsingHes}. Additionally, since $b \in \K$, so it follows from the \Cref{def:clc} that $b^{t}Q_ab=D_{b}D_{b}f(a) >0$. Then the result follows from \Cref{lemma:negsemdef}. \\
(b) This follows from (a) by taking limits and using \Cref{remark:int}. 
\end{proof}

\begin{lemma} \label{lemma:negdeter}
   Let $f \in \R[x]_n^d$ be a nonzero polynomial, $\K$ be a proper convex cone in $\R^{n}$, and $Q_a:=H_{f}(a)$ for $a \in \K$.
    \begin{enumerate}[(a)]
   \item  If $f \in \inter \CL^{d}(\K)$, the restriction of $Q_a$ on a plane spanned by $b,c \in  \K$ is a positive matrix with negative determinant.
   \item If $f \in \CL^{d}(\K)$, the restriction of $Q_a$ on a plane spanned by $b,c \in \K$ is a nonnegative matrix with nonpositive determinant.
   \end{enumerate}
\end{lemma}
\begin{proof}
$(a)$  By \Cref{lemma:nonsingHes}, $Q:=H_f(a)$ is nonsingular and has exactly one positive eigenvalue. Then the rest follows from \Cref{remark:defequiv} and the restriction of the Hessian to the plane spanned by $b, c \in \K$, outlined in \eqref{eq:negsemdef}. 

$(b)$ Since $\CL^{d}(\K)$ is the closure of its interior (See \Cref{remark:int}), the result follows from part $(a)$ by taking limits.
\end{proof}

\subsection{The sum of two Lorentzian polynomials}

The sum of two log-concave polynomials may not be log-concave.
The subsequent result provides a sufficient condition for the sum of two $\K$-Lorentzian polynomials to also be $\K$-Lorentzian. This condition is analogous to the one provided in \cite[Lemma 3.3]{Cynthialog3} for the sum of two log-concave polynomials with nonnegative coefficients to be log-concave on the nonnegative orthant.

\begin{theorem} \label{them:sum}
    Let $f,g \in \SL^{d}(\K)$.
    If there exist vectors $b,c \in \K$ such that $D_bf=D_cg \not \equiv 0$, then $f+g$ is $\K$-Lorentzian.
\end{theorem}
\begin{proof}
 We prove this result by induction on the degree $d$. If $d = 1$, then $(f + g)(a)=f(a)+g(a)> 0$ for any $a \in \inter \K$. Thus, $f+g$ is $\K$-Lorentzian. 

Now, suppose $d \geq 2$ and fix $a \in \inter \K$. Let $f_1=D_a f$ and $g_1=D_a g$. Note that $f_1$ and $g_1$ are $\K$-Lorentzian polynomials, and additionally $D_b f_1=D_bD_a f=D_a D_b f=D_a D_c g= D_c D_a g= D_c g_1$. If we can show that $D_bf_1=D_cg_1 \not \equiv 0$ for all $a \in \inter \K$, then by the induction assumption we have $f_1+g_1=D_a(f+g)$ is $\K$-Lorentzian for all $a\in \inter \K$ and therefore $f+g$ is $\K$-Lorentzian.

Since the set $\SL^{d-1}(\K)$ is closed it follows that $D_bf$ is a nonzero polynomial in $\SL^{d-1}(\K)$. In particular, $D_b f$ is strictly positive on $\inter K$, and therefore $D_a D_b f(a)=(d-1) (D_b f)(a)>0$. Thus we see that $D_b f_1=D_b D_a f=D_a D_b f$ is not identically zero.
 \end{proof}

\section{\texorpdfstring{$\K$}{}-Lorentzian polynomials over self-dual cones} \label{sec:selfdual}
In this section, we outline an interesting connection $\K$-Lorentzian polynomials on a self-dual cone $\K$ and the generalized Perron-Frobenius theorem. We refer the reader to \cite{selfdual} for various classes of self-dual cones other than the nonnegative orthant, the second order cone, and the cone of positive semidefnite matrices.

\subsection{$\K$-irreducibility}
We now show an alternative characterization of the set of Lorentzian polynomials $\SL^{d}(\K)$ for a self-dual cone $\K$. We note that for a quadratic form $q$, the Hessian matrix $H_q$ is equal to twice the matrix $Q$ of $q$. Therefore our results can be equivalently stated in terms of the matrix $H_q$.

\begin{theorem} \label{them:clcbdd}
Let $f \in \R[x]_n^d$ be a form of degree $d \geq 2$ and $\K$ be a self-dual cone. Then $f\in \SL^d(\K)$ if and only if for all $a_1,\dots,a_{d-2} \in \inter \K$ the matrix $Q$ of the quadratic form  $q=D_{a_1}\dots D_{a_{d-2}}f$ satisfies the following two conditions:
    \begin{enumerate}[(a)]
        \item $Q$ either nonsingular and $\K$-irreducible, or singular and $\K$-nonnegative.
        \item $Q$ has exactly one positive eigenvalue.
        \end{enumerate}
        \end{theorem}
 \begin{proof} This follows from \Cref{def:pls} and \Cref{cor:qclcirre}. \end{proof}

Therefore, we have the following characterization of Lorentzian polynomials or, equivalently, CLC polynomials on the nonnegative orthant.

\begin{corollary}
  Let $f \in \R[x]_n^d$ be a form of degree $d \geq 2$. Then $f$ is Lorentzian or CLC polynomial on $\R^{n}_{\geq 0}$ if and only if for all $a_1,\dots,a_{d-2} \in  \R^n_{>0}$, the matirx $Q$ of the quadratic form $q=D_{a_1}\dots D_{a_{d-2}}f$ satisfies the following two conditions: 
    \begin{enumerate} 
        \item $Q$ is either nonsingular and irreducible, or singular and nonnegative.
        \item $Q$ has exactly one positive eigenvalue.
        \end{enumerate}
\end{corollary}

We now provide an alternative characterization of the interior $\iSL^{d}(\K)$ of the set of $\K$-Lorentzian polynomials on a self-dual cone $\K$ using $\K$-positivity.

\begin{theorem} \label{them:k-irre}
     Let $f \in \R[x]_n^d$ be a homogeneous polynomial of degree $d \geq 2$ and $\K$ be a self-dual cone. Then $f \in  \iSL^{d}(\K)$ if and only if for any $a_1,\dots,a_{d-2} \in \K$, the matrix $Q$ of the quadratic form $q=D_{a_1}\dots D_{a_{d-2}}f$ satisfies the following two conditions:
    \begin{enumerate}
        \item $Q$ is $\K$-positive and nonsingular.
        \item $Q$ has exactly one positive eigenvalue.
        \end{enumerate}
\end{theorem}
\begin{proof} 
This follows from \Cref{them:plsbdd} (see also observations in \Cref{lem:red1} and \Cref{cor:kpos}).
\end{proof}

\subsection{Extreme Rays}
Let $\K \subseteq \R^{n}$ be a self-dual cone and $f \in \R[x]_n^d$ be a homogeneous polynomial of degree $d$ in $n$ variables. By the conical Minkowski theorem, a proper convex cone $\K \subseteq \R^{n}$ is generated by its extreme rays. We recall that $\K$-CLC polynomials are a closed set in $\R_n^d[x]$, which is the closure of its interior, which consists of strictly $\K$-CLC polynomials. Therefore, we may ask whether taking directional derivatives with respect to extreme rays is sufficient to guarantee that a polynomial is $\K$-CLC. First we demonstrate that in general, testing $\K$-CLC conditions on extreme rays of $\K$ is not sufficient. 
\begin{example}
Consider the polynomial $f=x_1(x_1x_2+x_1x_3+x_2x_3)+x_4^3$. Then over the nonnegative orthant $D_{e_i}f$ are $\K$-Lorentzian polynomials, but $D_{\1}f$ is not where $\1$ denotes all-ones vector. Consequently, $f$ is not a $\K$-Lorentzian polynomial over the nonnegative orthant.
\end{example}

However, with additional assumptions on the polynomial $f \in \R[x]_n^d$, testing $\K$-CLC conditions on the extreme rays of a self-dual cone $\K$ will be sufficient.

\begin{lemma} \label{prop:irredzero}
Let $f \in \R[x]_n^d$ be a form of degree $d\geq 2$. Let $\K \subset \R^{n}$ be a self-dual cone and $E_{\K}$ denote the set of extreme rays of cone $\K$. If the Hessian matrix $H_{f}(v)$ is $\K$-positive for all $v \in E_{\K}$, then $D_{v_1}D_{v_2}f$ is not identically zero for all $v_1, v_2 \in E_{\K}$.
\end{lemma}
\begin{proof} 
If  $D_{v_1}D_{v_2}f \equiv 0$ then we have  $v_2^{t}H_f(x)v_1=0$ for all $x \in \R^{n}$. This contradicts $\K$-positivity of $H_{f}(v_i)$.
\end{proof}
\begin{remark}
The above result does not hold if we only require that $H_{f}(v)$ is $\K$-irreducible for all $v \in E_k$. For instance, consider $f=x_1x_2+x_1x_3$ with $\K$ the nonnegative orthant in $\R^3$. While the Hessian matrices are nonsingular and irreducible, we have $\partial_2\partial_3f$ is identically zero. 
\end{remark}
\begin{theorem}
     Let $f \in \R[x]_n^d$ be a form of degree $d \geq 3$, and $\K$ be a self-dual cone where $E_{\K}$ denotes the set of its extreme rays. Suppose that the Hessian $H_f(v)$ is nonsingular and $\K$-positive for all $v \in E_{\K}$. If $D_{v}f$ is $\K$-CLC for all $v\in E_{\K}$, then $f$ is $\K$-CLC.
\end{theorem}
\begin{proof}
Let $a \in \K$ and 
write $a = \sum_{i=1}^{m} c_{i}v_{i}$ for real positive $c_i$ and $v_i\in E_\K$. Then we have $D_af = \sum_{i=1}^{m} c_{i}D_{v_{i}}f$, and $D_{v_{j}}D_{v_{i}}f=D_{v_{i}}D_{v_{j}}f$. Now by \Cref{prop:irredzero} all $D_{v_{j}}D_{v_{i}}f$ are nonzero polynomials and we can apply \Cref{them:sum} to see that $D_a f$ is $\K$-CLC for all $a\in \K$. This implies that $f$ is $\K$-CLC by the equivalence between $\K$-Lorentzian and $\K$-CLC polynomials.
\end{proof}

Finally, we provide an alternative proof of another necessary and sufficient condition first proved in \cite[Theorem 8.15]{Leake} for $f \in \R[x]_n^d$ to be a $\K$-Lorentzian polynomial. We demonstrate that testing conditions in \Cref{def:pls} for directional derivatives along the extreme rays and one more additional point in the interior actually suffices.

\begin{lemma} \label{lemma:intextreme}
Let $f \in \R[x]_n^d$ be a form of degree $d \geq 2$, and let $\K$ be a proper convex cone where $E_{\K}$ denotes the set of its extreme rays, and $a\in \inter \K$. If $D_{a}f$ and $D_{v_{i}}f$ are $\K$-CLC, then the sum $D_{a+v}f$ for any $v \in E_{\K}$ is also $\K$-CLC.
\end{lemma}
\begin{proof} We observe that for any non-zero $\K$-CLC polynomial $g$ we have that $D_{a}g$ is a nonzero polynomial for all $a \in \inter \K$.
If  $D_{v}f$ is identically zero, then the statement follows, so we may assume that $D_{v}f$ is a nonzero polynomial. Since $D_{v}f$ is $\K$-CLC, we see that $D_{v}D_af=D_aD_{v}f$ is a nonzero polynomial. We now conclude that the sum $D_{v+a}f=D_v f+D_a f$ is $\K$-CLC by \Cref{them:sum}.
\end{proof}

\begin{theorem} \label{them:extreme}
Let $f \in \R[x]_n^d$ be a form of degree $d \geq 2$, and let $\K$ be a proper convex cone where $E_{\K}$ denotes the set of its extreme rays, and $a\in \inter \K$. Then $f$ is $\K$-Lorentzian if and only if for all $v_1,\dots,v_{d-2} \in E_{\K}$ and all $0 \leq k \leq d-2$, the quadratic form $D_{v_1}D_{v_2} \dots D_{v_{k}}D_{a}^{d-2-k}f$ is $\K$-Lorentzian.
\end{theorem}

\begin{proof} $\Rightarrow$ Since $\SL^m(\K)$ is closed for all $m\geq 0$ we see that conditions on directional derivatives in directions of extreme rays and $a$ are just a subset of conditions defining $\K$-Lorentzian polynomials.

$\Leftarrow$  
We induct on the degree $d$ of $f$. The base case $d=2$ is trivially true since we do not take any directional derivatives.

For $d \geq 3$ we assume that the result is true up to degree $d-1$. It suffices to show that $D_b f \in \SL^{d-1}(\K)$ for all $b \in \inter \K$. We can write any $b\in \inter \K$ as a positive linear combination of $a$ and extreme rays $v_i$, where the coeficient of $a$ is strictly positive: $b=\lambda a+ \sum_{i=1}^{m}c_i v_i$ with $\lambda, c_i >0$, and $v_{i} \in E_{\K}$. 
 
Now we iteratively apply \Cref{lemma:intextreme} to $\lambda a+c_1v_1$, $(\lambda a+c_1v_1) +c_2v_2$, $(\lambda a+c_1v_1 +c_2v_2)+c_3v_3$ and so on, since $\lambda a+ \sum c_i v_i$ always lies in the interior of $\K$.

\end{proof}

An immediate corollary provides a characterization of Lorentzian polynomials which is first proved in \cite[Corollary 8.16]{Leake}.

\begin{corollary}
    Let $f \in \R[x]_n^d$ be a form of degree $d \geq 2$ and $\1$ denote the all-ones vector. Then $f$ is Lorentzian if and only if for all $i_1,\dots,i_{d-2} \in [n]$ and all $0 \leq k \leq d-2$, the quadratic polynomial $\partial _{i_1} \dots \partial_{i_k} D^{d-2-k}_{\1}f$ is Lorentzian.
\end{corollary}

For a self-dual cone $\K$, using \Cref{them:clcbdd}, we have the following result.
\begin{theorem} \label{clcextreme}
Let $f \in \R[x]_n^d$ be a form of degree $d \geq 2$ and $\K$ be a self-dual cone where $E_{\K}$ denotes the set of its extreme rays. Fix $a \in \inter \K$. Then $f$ is $\K$-Lorentzian if and only if for all $v_1,\dots,v_k \in E_{\K}$ where $0 \leq k \leq d-2$, the matrix $Q$ of the quadratic form $q=D_{v_1}D_{v_2} \dots D_{v_{k}}D_{a}^{d-2-k}f$ satisfies the following two conditions:
    \begin{enumerate}
    \item  $Q$ either nonsingular and $\K$-irreducible, or singular and $\K$-nonnegative. 
        \item $Q$ has exactly one positive eigenvalue.
        \end{enumerate}
\end{theorem}

As a consequence we provide another alternative necessary and sufficient conditions for Lorentzian polynomials.
 \begin{corollary}
Let $\1 \in \R^{n}_{>0}$ denote the all-ones vector.
$f$ is Lorentzian over $\R^{n}_{\geq 0}$ if and only if  for all $i_1,\dots,i_{d-2} \in [n]$ and $a \in \R^{n}_{>0}, 0 \leq k \leq d-2$, the matrix $Q$ of the quadratic form $q=\partial_{i_1} \dots \partial_{i_k} D_{\1}^{d-2-k}f$ satisfies the following two conditions:
 \begin{enumerate}
 \item $Q$ is either nonsingular and irreducible, or singular and nonnegative. 
        \item $Q$ has exactly one positive eigenvalue.
        \end{enumerate}
\end{corollary} 

    Note that it is easy to check irreducibility of a nonnegative matrix due to the following result, see \cite{Brualdiirrematrices}.

\begin{theorem}
An $n \times n$ nonnegative matrix $A$ is permutation equivalent to an irreducible matrix if and only if $A$ contains no zero row or zero column.
\end{theorem}

\bibliographystyle{alpha}
\bibliography{main}
\end{document}